\documentclass[12pt,reqno]{amsart}

\usepackage{a4wide,amscd,amsmath,amssymb,amsthm}
\usepackage[all]{xy}

\def\quo#1#2{\raisebox{.5ex}{{#1}\!}/\raisebox{-.5ex}{\!{#2}}}

\theoremstyle{plain}
\newtheorem{thm}{Theorem}[section]

\newtheorem{lem}[thm]{Lemma}

\newtheorem{cor}[thm]{Corollary}

\theoremstyle{definition}
\newtheorem{dfn}[thm]{Definition}

\def\M{\mathcal M}
\def\N{\mathcal N}
\def\I{\mathcal I}
\def\J{\mathcal J}

\def\Z{\ensuremath{\mathbb{Z}}}

\begin{document}

\title{Self--Maps of the Product of Two Spheres Fixing the Diagonal}

\date{\today}

\author{Hans-Joachim Baues and Beatrice Bleile}

\address{Max Planck Institut f{\"u}r Mathematik\\
Vivatsgasse 7\\
D--53111 Bonn, Germany}

\email{baues@mpim-bonn.mpg.de}

\address{Second author's home institution: School of Science and Technology\\
University of New England\\
NSW 2351, Australia}

\email{bbleile@une.edu.au}

\begin{abstract}
We compute the monoid of essential self--maps of $S^n \times S^n$ fixing the diagonal. More generally, we consider products $S \times S$, where $S$ is a suspension. Essential self--maps of $S \times S$ demonstrate the interplay between the pinching action for a mapping cone and the fundamental action on homotopy classes under a space. We compute examples with non--trivial fundamental actions.
\end{abstract}


\maketitle


%
%
%
%
%
\section{Introduction}
This paper investigates self--maps of $S^n \times S^n$ fixing the diagonal $\Delta: D = S^n \rightarrow S^n \times S^n$. To be more precise, we consider the set, $[ S^n \times S^n, S^n \times S^n]^{\Delta}$, of homotopy classes under $D$, and the function
\[\varphi_n: [ S^n \times S^n, S^n \times S^n]^{\Delta} \longrightarrow [ S^n \times S^n, S^n \times S^n]^{\ast},\]
which takes the homotopy class of a map relative $D$ to the homotopy class of the same map relative the base point $\ast \in D$. There is a fundamental action of $F = \pi_{n+1}(S^n) \oplus \pi_{n+1}(S^n)$ on $[S^n \times S^n, S^n \times S^n]^{\Delta}$ such that $\varphi_n(f) = \varphi_n(g)$ if and only if there is an $\alpha \in F$ such that $f+\alpha = g$, see Section \ref{fundact}.

\begin{thm}
The function $\varphi_n$ is injective, that is, the fundamental action is trivial for $n \geq 1$.
\end{thm}

In other words, given a self--map $F$ of $S^n \times S^n$ with $F \Delta \simeq \Delta$, there is a homotopy $F \simeq G$ with $G \Delta = \Delta$ and the homotopy class of $G$ is uniquely determined by $F$.

More generally, given a suspension $S$ of a co--H--group with diagonal $\Delta: D = S \subset S \times S$ and a map $v: D = S \rightarrow U$, we consider the function
\[\varphi: [S \times S, U]^v \longrightarrow [S \times S, U]^{\ast},\]
where $[S \times S, U]^v$ is the set of homotopy classes of maps under $D$. In Section \ref{compute} we compute both sets in terms of Whitehead products and obtain a criterion for $\varphi$ to be injective.

The function $\varphi$ is not injective in general, that is, there are spaces $U$ and maps $v: S \rightarrow U$, for $S = S^2$, such that the fundamental action on $[S \times S, U]^v$ is non--trivial. 
 
We compute the orbits of the fundamental action for $S = S^n \times S^n$ in Sections \ref{compute} and \ref{S2timesS2}, apparently providing the first example in the literature where such orbits are computed for subspaces other than points. In particular, we obtain

\begin{thm}
Take a map $v: S^n \rightarrow U$. Then
\[[S^n \times S^n,U]^v = \bigcup_u \quo{$\pi_{2n}(U)$}{$I_u$},\]
where $u$ ranges over the set $\I$ of all $u = (u',u'') \in \pi_n(U) \times \pi_n(U)$ with $[u',u''] = 0$ and $u'+u''=v$. Let $w = u'' + (-1)^{n-1}u'$, so that $v = w + (1 + (-1)^n)u'$. Then
\begin{eqnarray*}
I_u & = & \{ [\alpha,w] \ | \ \alpha \in \pi_{n+1}(U) \} \quad \text{and} \\
J_u & = &  \{ [\alpha,w] + [\gamma, u'] \ | \ \alpha, \gamma \in \pi_{n+1}(U) \}.
\end{eqnarray*}
The orbits of the fundamental action are given by the quotient groups $\quo{$J_u$}{$I_u$}$ acting on $\quo{$\pi_{2n}(U)$}{$I_u$}$. Thus $\varphi$ is injective if and only if $I_u = J_u$ for all $u \in \I$.
\end{thm}

To determine the monoid $[S^n \times S^n, S^n \times S^n]^{\Delta}$, consider the monoid $\N$ of $(2 \times 2)$--matrices over $\Z$ given by
\[\N = \{ \begin{bmatrix} a' & a''\\ b' & b'' \end{bmatrix} \ | \ a'+a'' = 1, b'+b'' = 1 \} \subset \text{End}(\Z \oplus \Z).\]
Let $\M$ be the submonoid of matrices with $a', a'', b', b'' \in \{0, 1\}$. There are four canonical self--maps of $S^n \times S^n$ which fix the diagonal, namely the identity, $I$, the interchange map, $T$, $P' = \Delta \circ \text{pr}_1$ and $P'' = \Delta \circ \text{pr}_2$, where $\text{pr}_i: S^n \times S^n \rightarrow S^n$ is the projection onto the $i$--th factor for $i = 1, 2$. We obtain the multiplication table
\[\begin{array}{|c||c|c|c|c|}
\hline 
\phantom{x} & I & T & P' & P'' \\ \hline \hline
I & I & T & P' & P'' \\ \hline 
T & T & I & P' & P'' \\ \hline 
P' & P' & P'' & P' & P'' \\ \hline 
P'' & P'' & P' & P' & P'' \\ \hline
\end{array}\]
and identify the monoid formed by $I, T, P'$ and $P''$ with the monoid $\M$. Let $\eta_{n+1} \in \pi_{n+1}(S^n)$ be the Hopf element, $i_n \in \pi_n(S^n)$ the identity and $[\eta_{n+1}, i_n] \in \pi_{2n}(S^n)$ the Whitehead product. We know that $\pi_3(S^2) = \Z$ and $\pi_{n+1}(S^n) = \Z_2, i \geq 3$, are generated by $\eta_{n+1}$. Moreover, for small $n$ the Whitehead product satisfies
\[\begin{array}{|c||c|c|c|c|}
\hline 
n & 2 & 3 & 4 & 5 \\ \hline
[ \eta_{n+1}, i_n ] & 0 & 0 & \neq 0 & \neq 0 \\ \hline
\end{array},\]  
see \cite{Toda}. We define the abelian group $V_n$ by
\[ V_n = \quo{$\pi_{2n}(S^n)$}{$[\eta_{n+1}, i_n]$}.\]
If $n$ is odd, $V_n \oplus V_n$ is an $\N$--bimodule. Namely, for $(x,y) \in V_n \oplus V_n$, the left action of $\alpha = \begin{bmatrix}a' & a'' \\ b' & b'' \end{bmatrix} \in \N$ is given by
\[\alpha (x,y) = (a'x+a''y,b'x+b''y)\]
and the right action is given by
\[(x,y) \alpha = (a'b'' + (-1)^n b'a'')(x,y).\]
If $n$ is even, $V_n \oplus V_n$ is an $\M$--bimodule by the same formul\ae. We define the monoid $\M_n = \M \times(V_n \oplus V_n)$ by the multiplication
\[(m,(x,y))\circ (m',(x',y')) = (mm',(m(x',y') + (x,y)m'),\]
that is, $\M_n$ is a split linear extension of $\M$.

\begin{thm}\label{computation}
The set $[ S^n \times S^n, S^n \times S^n]^{\Delta}$ together with composition of maps is a monoid isomorphic to $\M_n$, if $n$ is even. If $n$ is odd, the monoid $\N_n = [S^n \times S^n, S^n \times S^n]^{\Delta}$ is a linear extension of $\N$ by the bimodule $V_n \oplus V_n$, that is, there is a surjection $\pi: \N_n \rightarrow \N$ of monoids and a free action $+$ of $V_n \oplus V_n$ on $\N_n$, such that the linear distributivity law holds, that is,
\[(m+(x,y))\circ(m'+(x',y')) = m\circ m' + m(x',y') + (x,y)m',\]
and $\pi(m) = \pi(m')$ if and only if there is $(x,y) \in V_n \oplus V_n$ with $m + (x,y) = m'$ for $m, m' \in \N_n$.
\end{thm}

For $n$ odd it remains an open question whether the linear extension $\N_n$ splits. Here $\N_n$ splits if and only if the cohomology class $[\N_n] \in H^3(\N,V_n \oplus V_n)$ represented by $\N_n$ is trivial, see \cite{BauesWirsching}.

For the proof of Theorem \ref{computation} we use the fact that the Whitehead product $[i_n,i_n]$ has infinite order if $n$ is even, is trivial for $n = 1, 3, 7$ and otherwise an element of order $2$. Moreover, we use the realizability conditions for $u = (u',u'')$ in Corollary \ref{conditions}, where $u'+u'' =1$ implies that either $u'$ or $u''$ must be even.

Theorem \ref{computation} was proved for $n =2$ by different methods in \cite{BauesBleile}. The special case motived the authors to consider the general case in this paper.

\section{The Fundamental Action}\label{fundact}
Let $D$ be a space and let $i: D \rightarrowtail X$ be a cofibration. Given a map $u: D \rightarrow U$, we consider maps $f: X \rightarrow U$ \emph{under $D$}, that is, maps with $fi=u$. Two maps $f, g: X \rightarrow U$ under $D$ are \emph{homotopic relative $D$}, if there is a homotopy $H: f \simeq g$, such that for each $t, 0 \leq t \leq 1$, the map $H_t$ is also a map under $D$. Let $[X, U]^D = [X, U]^u$ be the set of homotopy classes relative $D$ of maps under $D$. For $D = \ast$ a point, the set $[X, U]^{\ast}$ is the usual set of homotopy classes of base point preserving maps. Given a cofibration $E\rightarrowtail D$, the forgetful map
\[\varphi: [X, U]^v \longrightarrow [X, U]^{E}\]
takes the homotopy class $[f]$ relative $D$ to the homotopy class $[f]$ relative $E$. The image of $\varphi$ is the subset of all elements $[g] \in [X, U]^{E}$ with $gi \simeq u$ relative $E$. Let $\Sigma_{E} D$ be the pushout of $S^1 \times E \twoheadrightarrow E$ and $S^1 \times E \rightarrowtail S^1 \times D$. Then $[\Sigma_{E}D, U]^v$ is a group acting on $[X, U]^v$ via the \emph{fundamental action $+$}, given by the homotopy extension property of the cofibration $D \rightarrowtail X$.  By II (5.17) in \cite{AH}, $\varphi(f) = \varphi(g)$, for $f, g \in [X, U]^v$, if and only if there is an $\alpha \in [\Sigma_{E}D, U]^v$ such that $f+\alpha = g$.

In general, the fundamental action is non--trivial. For example, if $E = \emptyset$, the empty set, and $D = \ast$, the point, then the fundamental action is the action of the fundamental group. If $X = K(G,1)$ and $U = K(H,1)$ are Eilenberg--MacLane spaces, then $[X, U]^{\ast} = \text{Hom}(G,H)$ and $\pi_1(U) = H$ acts via $(\varphi + \alpha)(g) = - \alpha + \varphi(g) + \alpha$ for $\alpha \in H$ and $\varphi \in \text{Hom}(G,H)$.
\section{The Pinching Action}
Choosing a closed ball $B^{2n}$ in the complement, $S^n \times S^n \setminus \Delta(S^n)$, of the diagonal, we obtain the \emph{pinching map}, $\mu: S^n \times S^n \rightarrow S^n \times S^n \vee S^{2n}$, by identifying the boundary of the ball to a point. The map $\mu$ induces the \emph{pinching action} of the group $\pi_{2n}(U)$ on the set $[S^n \times S^n, U]^u$, where $u: D=S^n \rightarrow U$. This action commutes with the fundamental action of Section \ref{fundact}. Since $D=S^n$ is a suspension, there is a homotopy equivalence, $\Sigma_{\ast}D \simeq S^{n+1} \vee D$, under $D$, and the fundamental action on $[S^n \times S^n, U]^u$ is an action of the group $[\Sigma_{\ast}D, U]^u = \pi_{n+1}(U)$. The pinching action and the fundamental action define an action of $\pi_{2n}(U) \oplus \pi_{n+1}(U)$ on $[S^n \times S^n, U]^u$. Putting $U = S^n \times S^n$ and denoting the homology functor by $H_n$, we obtain

\begin{lem}\label{combaction}
Take $f, g \in [ S^n \times S^n , S^n \times S^n]^{\Delta}$. Then $H_n(f) = H_n(g)$ if and only if there is an $\alpha \in W_n = \pi_{2n}(S^n \times S^n) \oplus \pi_{n+1}(S^n \times S^n)$ with $f+\alpha = g$. \end{lem}

Lemma \ref{combaction} follows from equation (\ref{Z}) in Section \ref{compute} which is devoted to the computation of the isotropy groups of the action of $\pi_{2n}(U) \oplus \pi_{n+1}(U)$ on $[S^n \times S^n, U]^u$.

\section{Computation of the Isotropy Groups}\label{compute}
Let $S = \Sigma T$ be the suspension of a pointed space and assume that the diagonal $\Delta: D = S \subset S \times S$ is a cofibration. There is a homotopy $H: \Delta \simeq \mu$ from the diagonal $\Delta$ to the comultiplication $\mu: D = S \rightarrow S \vee S \subset S \times S$. Let $Z$ be the mapping cylinder of $\mu$. Then $H$ yields a map $\overline H: Z \rightarrow S \times S$ under $D$, which is a homotopy equivalence in $\textbf{Top}^{\ast}$. By Corollary II 2.21 in \cite{AH}, $\overline H$ is also a homotopy equivalence in $\textbf{Top}^D$ and hence, for a map $v: D = S\rightarrow U$,
\begin{equation}\label{Z}
[Z, U]^v = [S \times S, U]^v.
\end{equation}
We consider the diagram
\[\xymatrix{
&& Z = C_f \ar[drr]^w && \\
A \ar[rr]^f && B = C_h \ar[rr]^u \ar[u]^{i_f} && U \\
Q \ar[rr]^h && D \ar[urr]_v \ar[u] && }\]
which corresponds to diagram II (13.1) in \cite{AH}. We are interested in the special case where $D = S, Q = T \vee T$ and $h$ is the trivial map. Then $B = S' \vee S'' \vee D$ is a $1$--point union of three copies of $S = S', S'', D$, and we denote the inclusions of $S$ by $e', e''$ and $e$, respectively. Further, we put $A = \Sigma(T \wedge T) \vee S$ and
\begin{eqnarray*}
f_1 & = & f|_{\Sigma(T\wedge T)} = [e',e''], \quad \text{the Whitehead product},\\
f_2 & = & f|_{S} = -e + e' + e'', \quad \text{the sum of inclusions}.
\end{eqnarray*}
It is known that the mapping cone $C_f$ coincides with $Z$, see for example (0.3.3) in \cite{OT}. For $S = S^n$, the coaction on $C_f$ yields an action corresponding to the action in Lemma \ref{combaction}. The map $u$ in the diagram is an extension of $v$ and thus determined by $u', u'': S \rightarrow U$. In order to apply II(13.10) in \cite{AH}, we recall the definition of the \emph{difference element} of a map $f: A \rightarrow C_h$, where $C_h$ is the mapping cone of a map $h: Q \rightarrow D$. The inclusions $i_1:  \Sigma Q \rightarrow \Sigma Q \vee C_h$ and $i_2: C_h \rightarrow \Sigma Q \vee C_h$ yield the map $i_2+i_1: C_h \longrightarrow \Sigma Q \vee C_h$ and 
\[\nabla f = -f^{\ast}(i_2) + f^{\ast}(i_2 + i_1) : A \longrightarrow \Sigma Q \vee C_h.\]
Note that $\nabla f$ is trivial on $C_h$, that is, $(0,1)_{\ast} \nabla f = 0$. An extension $w$ of $u$ gives rise to $w^+: [\Sigma A,U] \rightarrow [C_f,U]^v, \alpha \mapsto w+\alpha$, where the action $+$ is given by the pinching map of the mapping cone $C_f$. Now II(13.10) in \cite{AH} yields

\begin{thm}\label{exactsequence}
There is an exact sequence
\[\xymatrix{
[\Sigma^2Q,U] \ar[r]^{\nabla(u,f)} & [\Sigma A,U] \ar[r]^{w^+} & [C_f,U]^v \ar[r]^{i_f^{\ast}} & [B,U]^v \ar[r]^{f^{\ast}} & [A,U],}\]
where $\nabla(u,f)(\beta) = (E\nabla f)^{\ast}(\beta,u)$ and $E$ is the partial suspension. 
\end{thm}

Special cases of Theorem \ref{exactsequence} for $D = \ast$ correspond to results by Barcus and Barratt \cite{BarcusBarratt} and Rutter \cite{Rutter}.

For the computation of $\nabla(u,f)$ we must consider $\nabla f: A \rightarrow \Sigma Q_1 \vee \Sigma Q_2 \vee D$, where $Q_1 = Q_2 = T \vee T$ and $E \nabla f: \Sigma A \rightarrow \Sigma^2 Q_1 \vee \Sigma Q_2 \vee D$. The inclusions $e'$ and $e''$ yield the corresponding inclusions $e'_1, e'_2, e''_1, e''_2$ of $S$ in $\Sigma Q_1 \vee \Sigma Q_2 \vee D$ and the inclusions $\Sigma e'_1, \Sigma e''_1$ of $\Sigma S$ in $\Sigma^2 Q_1 \vee \Sigma Q_2 \vee D$, so that
\begin{eqnarray*}
\nabla f_1 & = & -[e'_2, e''_2] + [e'_2+e'_1, e''_2+e''_1], \\
\nabla f_2 & = & -(-e+e'_2+e''_2) + (-e+e'_2+e'_1+e''_2+e''_1).\\
\end{eqnarray*}
If $T$ is a co--H--group, the Whitehead product is bilinear and we obtain
\begin{eqnarray*}
\nabla f_1 
& = & [e'_1,e''_2] + [e'_2,e''_1] + [e'_1,e''_1] \\
& = & [e'_1,e''_2] - (\Sigma \tau)^{\ast} [e''_1,e'_2] + [e'_1,e''_1],
\end{eqnarray*}
where $\tau: T \wedge T \rightarrow T \wedge T$ is the interchange map. Then
\[E\nabla f_1 = [\Sigma e'_1,e''_2] - (\Sigma^2\tau)^{\ast}[\Sigma e''_1,e'_2],\]
see (3.1.11) \cite{OT}. Moreover, $\nabla f_2 = e'_1 + e''_1$, so that
\[E \nabla f_2 = \Sigma e'_1 + \Sigma e''_1.\]
By the definition of $\nabla(u,f)$ in Theorem \ref{exactsequence}, we thus obtain
\[\begin{array}{ll}
\nabla(u,f_1) :  [\Sigma^2(T \vee T),U] \rightarrow [\Sigma^2(T \wedge T),U], & \nabla(u,f_1)(\alpha, \beta) = [\alpha,u''] - (\Sigma^2 \tau)^{\ast}[\beta,u'],\\
\nabla(u,f_2) :  [\Sigma^2(T \vee T),U] \rightarrow [\Sigma^2 T,U], & \nabla(u,f_2)(\alpha, \beta) = \alpha + \beta.
\end{array}\]
We leave it to the reader to compute these functions when $T$ is not a co--H--group.

For the suspension $S = \Sigma T$ with diagonal $\Delta: S \subset S \times S$ and a map $v: S \rightarrow U$ as above, consider the commutative diagram

\begin{equation}\label{phi}
\xymatrix{
[Z, U]^v = [S \times S, U]^v \ar[rr]^{\varphi} \ar[dr]_{p_{\Delta}} && [S \times S,U]^{\ast} \ar[dl]^{p}\\
&[S \vee S, U]^{\ast}.}
\end{equation}
The group $[\Sigma^2 T \wedge T, U]$ acts on $[S \times S, U]^{\ast}$ and the group $[\Sigma^2 T \wedge T, U] \oplus [\Sigma^2 T,U]$ acts on $[Z, U]^v$. For $f, g \in [S \times S, U]^{\ast}$ we obtain $p(f) = p(g)$ if and only if there is an element $\alpha \in [\Sigma^2 T \wedge T, U]$ with $f+\alpha = g$. For $f, g \in [Z, U]^v$ we obtain $p_{\Delta}(f) = p_{\Delta}(g)$ if and only if there is an element $(\alpha,\beta) \in [\Sigma^2 T \wedge T, U] \oplus [\Sigma^2 T,U]$ with $f+(\alpha,\beta) = g$. An element $u = (u',u'') \in [S \vee S, U]^{\ast}$ is in the image of $p$ if and only if $[u',u''] = 0$, and then the isotropy group of the orbit $p^{-1}(u)$ is the image of $\nabla(u,f_1)$, see (3.3.15) in \cite{OT}. Moreover, $u$ is in the image of $p_{\Delta}$ if and only if $u'+u'' = v$ and $[u',u''] = 0$, and then the isotropy group of the orbit $p_{\Delta}^{-1}(u)$ is the image of $\nabla(u,f_1)+\nabla(u,f_2)$.

\begin{thm}\label{compagain}
Let $S = \Sigma T$ be the suspension of a co--H--group $T$ and let $v: S \rightarrow U$ be a map. Choosing a representative in each orbit, we obtain the bijection
\[ [S \times S, U]^v = \bigcup_u \quo{$[\Sigma^2 T \wedge T, U]$}{$I_u$},\]
where $u$ ranges over the set $\I$ of all $u = (u',u'') \in [S,U] \times [S,U]$ with $[u',u''] = 0$ and $u'+u'' = v$, and
\[I_u = \{ [\alpha,u''] + (\Sigma^2 \tau)^{\ast}[\alpha,u'] \ | \ \alpha \in [\Sigma^2 T, U] \}.\]
\end{thm}

\begin{proof}
Surjectivity of $\nabla(u,f_2)$ yields the commutative diagram
\[\xymatrix{
\ker \nabla(u,f_2) \ \ar@{>->}[d] \ar[rr]^{\nabla(u,f_1)} && [\Sigma^2 T \wedge T, U] \ar@{>->}[d] \ar@{->>}[rr] && 
\quo{$[\Sigma^2 T \wedge T, U]$}{$I_u$} \ar@{>->}[d]^{\cong} \\
[\Sigma^2 Q, U] \ar@{->>}[d]_{\nabla(u,f_2)} \ar[rr]^{\nabla(u,f)} \ar[rr] && [\Sigma A, U] \ar@{->>}[d] \ar@{->>}[rr] && \text{coker} \nabla(u,f) \ar@{->>}[d] \\
[\Sigma S,U] \ar@{=}[rr] && [\Sigma S,U] \ar[rr] && 0.}\]
of short exact columns and exact rows, where $I_u = \nabla(u,f_1)(\ker(\nabla(u,f_2)))$. The formula for $I_u$ follows from the computation of $\nabla(u,f_1)$ and $\nabla(u,f_2)$. 
\end{proof}

Theorem \ref{compagain} corresponds to (3.3.15) in \cite{OT}, where $[S \times S,U]^{\ast}$ is computed by the formula
\[[S \times S,U]^{\ast} = \bigcup_u \quo{$[\Sigma^2T\wedge T,U]$}{$J_u$},\]
where $u$ ranges over the set $\J$ of all $u = (u',u'') \in [S \vee S,U]$ with $[u',u''] = 0$ and
\[J_u = \{ [\alpha,u''] - (\Sigma^2 \tau)^{\ast}[\beta,u'] \ | \ \alpha, \beta \in [\Sigma^2 T, U] \}.\]
Moreover, the diagram
\[\xymatrix{
[S \times S, U]^v \ar@{=}[r] \ar[d]_{\varphi} & \bigcup_{u \in \I} \quo{$[\Sigma^2 T \wedge T,U]$}{$I_u$} \ar[d] \\
[S \times S, U]^{\ast} \ar@{=}[r] & \bigcup_{u \in \J} \quo{$[\Sigma^2T \wedge T,U]$}{$J_u$}
}\]
commutes, where $\I \subset \J$ and the arrow on the right hand side is induced by the identity. Thus we obtain

\begin{cor}
The orbits of the fundamental action are given by the quotient groups $\quo{$J_u$}{$I_u$}$ acting on $\quo{$[\Sigma^2T\wedge T,U]$}{$I_u$}$. Thus $\varphi$ is injective if and only if $I_u = J_u$ for all $u \in \I$.
\end{cor}

For the special case of a sphere $S = S^n$, we obtain

\begin{cor}\label{conditions}
Take a map $v: S^n \rightarrow U$. Then
\[[S^n \times S^n,U]^v = \bigcup_u \quo{$\pi_{2n}(U)$}{$I_u$},\]
where $u$ ranges over the set $\I$ of all $u = (u',u'') \in \pi_n(U) \times \pi_n(U)$ with $[u',u''] = 0$ and $u'+u''=v$. Let $w = u'' + (-1)^{n-1}u'$, so that $v = w + (1 + (-1)^n)u'$. Then
\begin{eqnarray*}
I_u & = & \{ [\alpha,w] \ | \ \alpha \in \pi_{n+1}(U) \} \quad \text{and} \\
J_u & = &  \{ [\alpha,w] + [\gamma, u'] \ | \ \alpha, \gamma \in \pi_{n+1}(U) \}.
\end{eqnarray*}
\end{cor}

\begin{cor}
Let $v = \text{id}_{S^n}: S^n \rightarrow S^n$ be the identity. Then the fundamental action on $[S^n \times S^n, S^n]^v$ is trivial.
\end{cor}

\begin{proof}
Take $u = (u',u'') \in \Z \oplus \Z = [S^n \vee S^n, S^n]$ with $u' u''[i_n,i_n] = [u',u''] = 0$ and $u'+u'' = v = 1$. Then $I_u$ is the subgroup of all elements of the form $(u''+(-1)^{n+1}u')[\eta_{n+1},i_n]$ and $J_u$ is the subgroup generated by $u''[\eta_{n+1}, i_n]$ and $u'[\eta_{n+1},i_n]$. As $u'+u'' = 1$ implies that either $u'$ or $u''$ must be odd and $[\eta_{n+1},i_n] = 0$ for $n = 2$ and is an element of order $2$ otherwise, we conclude that $I_u = J_u$.
\end{proof}

\section{Computation of $I_u$ and $J_u$ for $S^2 \times S^2$}\label{S2timesS2}
First recall Whitehead's quadractic functor $\Gamma$. A function $\eta: \pi_2 \rightarrow \pi_3$ between abelian groups is \emph{quadratic} if $\eta(-a) = \eta(a)$ and $\pi_2 \times \pi_2 \rightarrow \pi_3, (a,b) \mapsto \eta(a+b) - \eta(a) - \eta(b) = [a,b]_{\eta}$ is bilinear. There is a \emph{universal quadratic map} $\gamma: \pi_2 \rightarrow \Gamma(\pi_2)$, such that there is a unique homomorphism $\eta^{\square}: \Gamma(\pi_2) \rightarrow \pi_3$ with $\eta = \gamma \eta^{\square}$. To define the $\Gamma$--torsion $\Gamma T(\pi_2)$, take a short free resolution $A_1 \stackrel{d}{\rightarrowtail} A_0 \twoheadrightarrow \pi_2$ and consider the sequence
\[ \xymatrix
{A_1 \otimes A_1 \ar[r]^-{\delta_2} & \Gamma(A_1) \oplus A_1 \otimes A_0 \ar[r]^-{\delta_1} & \Gamma(A_0),}\]
where $\delta_1 = (\Gamma(d), [d,1])$ and $\delta_2 = ([1,1], - 1 \otimes d)$.
Then
\[\Gamma T(\pi_2) = \quo{$\ker(\delta_1)$}{$\text{im}(\delta_2)$}.\]
Moreover, let $M(\eta)$ be the subgroup of $\pi_3 \otimes \quo{$\Z$}{$2$} \oplus \pi_3 \otimes \pi_2$ generated by
\begin{eqnarray*}
&&(\eta x) \otimes x \quad \text{and}\\
&& [x,y]'_{\eta} \otimes 1 + (\eta x) \otimes y + [y,x]'_{\eta} \otimes x,
\end{eqnarray*}
where $x, y \in \pi_2$ and $[x,y]'_{\eta} = \eta(x+y) - \eta(y)$. Putting 
\[\Gamma^2_2(\eta) = \quo{$(\pi_3 \otimes \quo{$\Z$}{$2$} \oplus \pi_3 \otimes \pi_2)$}{$M(\eta)$},\]
there is a short exact sequence of abelian groups
\begin{equation}\label{ses}
\xymatrix{
\Gamma^2_2(\eta) \, \ar@{>->}[r] & \Gamma_4(\eta) \ar@{->>}[r] & \Gamma T(\pi_2),}
\end{equation}
where $\Gamma_4(\eta)$ is the group $\Gamma_4 K(\eta,2)$ in (12.3.3) which is the natural element contained in $\text{Ext}(\Gamma T(\pi_2), \Gamma_2^2(\eta))$ computed in (11.3.4) and (11.1.21) in \cite{HTH}.

\begin{dfn}
Let $\pi_2, \pi_3, \pi_4, H_2, H_3, H_4$ and $H_5$ be abelian groups with $H_5$ free abelian and $\pi_2 = H_2$. A \emph{($5$--dimensional) $\Gamma$--sequence} is an exact sequence in the category of abelian groups
\begin{equation}\label{es}
\xymatrix{
H_5 \ar[r] & \Gamma_4(\eta) \ar[r] & \pi_4 \ar[r] & H_4 \ar[r] & \Gamma(\pi_2) \ar[r]^{\eta^{\square}} & \pi_3 \ar[r] & H_3 \ar[r] & 0.}
\end{equation}
A \emph{morphism $\varphi$ of $\Gamma$--sequences} is given by homomorphisms $H_i(\varphi): H_i \rightarrow H_i', \pi_i(\varphi): \pi_i \rightarrow \pi_i'$ and $\Gamma_4(\varphi): \Gamma_4(\eta) \rightarrow \Gamma_4(\eta')$ compatible with the exact sequences (\ref{ses}) and (\ref{es}). Here $(\Gamma(\pi_2(\varphi)), \pi_3(\varphi))$ is a morphism $\eta \rightarrow \eta'$ inducing a map $\Gamma^2_2(\eta) \rightarrow \Gamma^2_2(\eta')$.
\end{dfn}

The following result is proved in \cite{HTH}.

\begin{thm}\label{real}
There is a representable functor $W$ from the homotopy category of simply connected $5$--dimensional CW--complexes to the category of $5$--dimensional $\Gamma$--sequences. In other words, for every $\Gamma$--sequence (\ref{es}) there is a simply connected $5$--dimensional CW--complex $U$, such that the $\Gamma$--sequence $W(U)$ is isomorphic to (\ref{es}) in the category of $\Gamma$--sequences. We call $U$ a \emph{realization} of (\ref{es}).
\end{thm}

Note that (\ref{es}) is realizable, but does not determine the homotopy type of the realization $U$. Complete invariants classifying the homotopy type of $U$ are provided in (12.5.9) in \cite{HTH}. Theorem \ref{real} follows from (3.4.7) and (11.3.4) in \cite{HTH}, see also 7.10 in \cite{handbook}. The functor $W$ is an enrichment of Whitehead's Certain Exact Sequence \cite{Whitehead1}.

For $n = 2$ we now show that the groups $I_u$ and $J_u$ in Corollary \ref{conditions} depend on the $\Gamma$--sequence of $U$ only. Given a $\Gamma$--sequence (\ref{es}) and $y \in \pi_2$, let $[\pi_3,y] \subset \pi_4$ denote the image of 
\[\xymatrix{
\pi_3\otimes \langle y \rangle \ar[r] & \pi_3 \otimes \pi_2 \ar[r] & \Gamma^2_2(\eta)\  \ar@{>->}[r] & \Gamma_4(\eta) \ar[r] & \pi_4.}\]
By (11.3.5) in \cite{HTH}, $[\pi_3, y]$ corresponds to the Whitehead product.

\begin{thm}
Given a $\Gamma$--sequence (\ref{es}) and a realization $U$ of (\ref{es}), take $w, u' \in \pi_2 \cong \pi_2(U)$, such that $[u',u''] = 0$ for $u'' = w + u'$. Then
\begin{eqnarray*}
I_u & \cong & [\pi_3, w] \quad \text{and} \\
J_u & \cong & [\pi_3,w] + [\pi_3, u'].
\end{eqnarray*}
\end{thm}

In order to find examples with $I_u \neq J_u$, we choose appropriate $\Gamma$--sequences.

\begin{cor}
Consider a $\Gamma$--sequence (\ref{es}) with realization $U$ and elements $w, u' \in \pi_2$, with $\pi_2 = \pi_2' \oplus \langle w \rangle \oplus \langle u' \rangle$ and $\pi_3 \otimes \langle u' \rangle \neq 0$. Further, assume (\ref{es}) satisfies $\eta = 0$ and $H_5 = 0$. Then $[u', u''] = 0$ for $u'' = w +u'$,
\begin{eqnarray*}
I_u &  \cong & \pi_3 \otimes \langle w\rangle \quad \text{and} \\
J_u & \cong & \pi_3 \otimes ( \langle w \rangle \oplus \langle u' \rangle ), 
\end{eqnarray*}
so that $\quo{$I_u$}{$J_u$} \cong \pi_3 \otimes \langle u' \rangle \neq 0$. Hence the fundamental action on $[S^2 \times S^2, U]^v$ is non--trivial, where $v: S^2 \rightarrow U$ represents $u'+u'' = w + 2u'$. Moreover, the orbits of the fundamental action are of the form $\pi_3 \otimes \langle u' \rangle$, where the abelian group $\pi_3$ can be chosen arbitrarily and $\langle u' \rangle$ can be any cyclic group.
\end{cor}

\begin{proof}
As $H_5 = 0$, the map $\Gamma_4(\eta) \rightarrow \pi_4$  is injective. Therefore the map $\Gamma^2_2(\eta) \rightarrow \pi_4$ is also injective. Further, $\eta = 0$ implies $M(\eta) = 0$ and hence $\Gamma^2_2(\eta) = \pi_3 \otimes \quo{$\Z$}{$2$} \oplus \pi_3 \otimes \pi_2$. Thus $[\pi_3, w] \cong \pi_3 \otimes \langle w \rangle$ and $[\pi_3,w] + [\pi_3,u'] \cong \pi_3 \otimes (\langle w \rangle \oplus \langle u' \rangle)$.

\end{proof}


\begin{thebibliography}{99}

\bibitem{BarcusBarratt}
W.D. Barcus and M.G. Barratt, \emph{On the homotopy classification of the extensions of a fixed map}, Trans. Am. Math. Soc. 88, 57--77 (1958).

\bibitem{OT}
H.J. Baues, \emph{Obstruction Theory}, Lecture Notes in Mathematics, Springer Verlag, (1977).

\bibitem{AH}
H.J. Baues, \emph{Algebraic Homotopy}, Cambridge University Press, (1989).

\bibitem{HTH}
H.J. Baues, \emph{Homotopy Types and Homology}, Clarendon Press Oxford, (1996).

\bibitem{handbook}
H.J. Baues, \emph{Homotopy Types}, Handbook of Algebraic Topology, Elsevier (1995).

\bibitem{BauesBleile}
H.J. Baues and B. Bleile, \emph{Presentation of Homotopy Types under a Space}, (2010).

\bibitem{BauesWirsching}
H.J. Baues and G. Wirsching, \emph{The cohomology of small categories}, J. Pure Appl. Algebra 38, 187--211 (1985).

\bibitem{Rutter}
J.W. Rutter, \emph{A homotopy classification of maps into an induced fibre space}, Topology 6, 379--403 (1967).

\bibitem{Toda}
H. Toda, \emph{Generalized Whitehead Products and Homotopy Groups of Spheres}, Journal of the Institute if Polytechnics, Osaka City University, Vol. 3, No. 1--2, 44--82 (1952).

\bibitem{Whitehead1}
J.H.C. Whitehead, \emph{A certain exact sequence}, Ann. of Math. 52 (1950), 51--110.


\end{thebibliography}
\end{document}